\documentclass[preprint,3p,12pt,times]{elsarticle}

\journal{Elsevier}

\usepackage{amsmath}
\usepackage{amssymb}
\usepackage{mathrsfs}
\def\divv{\mathrm{d}}
\def\diff{\,\mathrm{d}}

\newdefinition{defn}{Definition}[section]
\newdefinition{rem}{Remark}
\newdefinition{exam}{Example}
\newtheorem{thm}{Theorem}[section]
\newtheorem{lem}[thm]{Lemma}

\newproof{proof}{Proof}

\renewcommand{\Re}{\mathop{\mathrm{Re}}\nolimits}
\renewcommand{\Im}{\mathop{\mathrm{Im}}\nolimits}

\DeclareMathOperator{\rme}{e}
\DeclareMathOperator{\imnum}{i}
\DeclareMathOperator{\arcsinh}{arcsinh}

\renewcommand{\pi}{\piup}

\newcommand{\textSE}[1]{\text{\tiny{\rm{SE$#1$}}}}
\newcommand{\textDE}[1]{\text{\tiny{\rm{DE$#1$}}}}
\newcommand{\textSEg}{\text{\scriptsize{\rm{SE}}}}
\newcommand{\textDEg}{\text{\scriptsize{\rm{DE}}}}
\newcommand{\SEt}[1]{\psi_{\textSE{#1}}}
\newcommand{\DEt}[1]{\psi_{\textDE{#1}}}

\newcommand{\SEtInv}[1]{\SEt{#1}^{-1}}
\newcommand{\DEtInv}[1]{\DEt{#1}^{-1}}

\newcommand{\domD}{\mathscr{D}}
\newcommand{\LC}{\mathbf{L}}

\newcommand{\Hone}{\mathbf{H}^1}

\numberwithin{equation}{section}

\begin{document}

\begin{frontmatter}

\title{Error estimates with explicit constants for the Sinc approximation
over infinite intervals\tnoteref{mytitlenote}}
\tnotetext[mytitlenote]{This work was partially supported by Grant-in-Aid for
Young Scientists (B) Number JP24760060.}


\author{Tomoaki Okayama}
\address{Graduate School of Information Sciences, Hiroshima City University\\
3-4-1, Ozuka-higashi, Asaminami-ku, Hiroshima 731-3194, Japan}
\ead{okayama@hiroshima-cu.ac.jp}


%

\begin{abstract}
The Sinc approximation is a function approximation formula
that attains exponential convergence for
rapidly decaying functions defined on the whole real axis.
Even for other functions, the Sinc approximation works accurately
when combined with a proper variable transformation.
The convergence rate has been analyzed for typical cases
including finite, semi-infinite, and infinite intervals.
Recently, for verified numerical computations,
a more explicit, ``computable'' error bound
has been given in the case of a finite interval.
In this paper, such explicit error bounds are derived
for other cases.
\end{abstract}


\end{frontmatter}



\section{Introduction}
\label{sec:introduction}

The ``Sinc approximation'' is a function approximation formula that can be expressed as
\begin{equation}
F(x)\approx \sum_{k=-M}^{N} F(kh) S(k,h)(x),\quad x\in\mathbb{R},
\label{eq:Sinc-approx}
\end{equation}
where $S(k,h)(x)$ is the so-called Sinc function, defined by
\[
 S(k,h)(x)=\frac{\sin[\pi(x/h - k)]}{\pi(x/h - k)},
\]
and $M$, $N$, $h$ are selected according to
the given positive integer $n$.
It is well known that
the approximation formula~\eqref{eq:Sinc-approx}
can converge \emph{exponentially}
when combined with an appropriate variable transformation.
Furthermore,
it has been shown that
the approximation formula~\eqref{eq:Sinc-approx} is nearly optimal
for functions in some Hardy spaces~\cite{stenger78:_optim,sugihara03:_near}.
Motivated by this fact,
many researchers have studied applications of the Sinc approximation
(see Stenger~\cite{stenger93:_numer,stenger00:_summar},
Lund--Bowers~\cite{lund92:_sinc},
Sugihara--Matsuo~\cite{sugihara04:_recen},
and references therein).

As stated above, to attain exponential convergence,
an appropriate variable transformation $t=\psi(x)$
is required.
In this regard, depending on the target interval $(a,\,b)$
and the function $f$,
Stenger~\cite{stenger93:_numer} considered the following four typical cases:
\begin{enumerate}
 \item $(a,\,b)=(-\infty,\,\infty)$ and $|f(t)|$ decays algebraically
as $t\to \pm\infty$,
 \item $(a,\,b)=(0,\,\infty)$ and $|f(t)|$ decays algebraically
as $t\to \infty$,
 \item $(a,\,b)=(0,\,\infty)$ and $|f(t)|$ decays (already) exponentially
as $t\to \infty$,
 \item The interval $(a,\,b)$ is finite.
\end{enumerate}
In all four cases, Stenger gave the concrete transformations to be employed:
\begin{align*}
\SEt{1}(x)=\sinh x,&&
\SEt{2}(x)=\mathrm{e}^x,&&
\SEt{3}(x)=\arcsinh(\mathrm{e}^x),&&
\SEt{4}(x)=\frac{b-a}{2}\tanh\left(\frac{x}{2}\right)+\frac{b+a}{2}.
\end{align*}
These are called conformal maps in the literature.
After applying the variable transformation $t=\SEt{i}(x)$,
we may set $F(x)=f(\SEt{i}(x))$ and use
the Sinc approximation~\eqref{eq:Sinc-approx}.
As a result, we have the following approximation for $f(t)$:

\begin{equation}
f(t) \approx \sum_{k=-M}^N f(\SEt{i}(kh))S(k,h)(\SEtInv{i}(t)),
\label{eq:SE-Sinc-approx}
\end{equation}
which is referred to as the ``SE-Sinc approximation'' in this paper.
Stenger~\cite{stenger93:_numer} demonstrated the exponential convergence
of the approximation~\eqref{eq:SE-Sinc-approx}
by giving error analyses
in the following form:
\begin{align}
 \sup_{t\in (a,\,b)}
\left|
f(t)
-\sum_{k=-M}^N f(\SEt{i}(kh))S(k,h)(\SEtInv{i}(t))
\right|
&\leq C \sqrt{n} \rme^{-\sqrt{\pi d \mu n}},
\label{leq:SE-Sinc-error}
\end{align}
where $\mu$ and $d$ are positive parameters of the analytic properties of $f$,
and $C$ is a constant independent of $n$.

In recent decades,
several authors~\cite{sugihara04:_recen,mori01:_doubl,tanaka09:_desinc}
have pointed out that
the performance of the Sinc approximation
can be improved by replacing the original variable transformations
with the following ones:
\begin{align*}
\DEt{1}(t)=\SEt{1}\left(\frac{\pi}{2}\sinh t\right),&&
\DEt{2}(t)=\SEt{2}\left(\frac{\pi}{2}\sinh t\right),&&
\DEt{3\dagger}(t)=\mathrm{e}^{t-\exp(-t)},&&
\DEt{4}(t)=\SEt{4}\left(\pi\sinh t\right).
\end{align*}
These are called the ``Double-Exponential (DE) transformations,''
and were originally introduced by Takahasi--Mori~\cite{takahasi74:_doubl}
for numerical integration.
In addition,
in case~3, another DE transformation
\begin{equation*}
\DEt{3}(t)=\log(1+\rme^{(\pi/2)\sinh t})
\end{equation*}
has been proposed~\cite{muhammad03:_doubl}
so that the inverse function can be explicitly written
by using elementary functions
(whereas $\DEt{3\dagger}(t)$ cannot).
The combination of the DE transformation and~\eqref{eq:Sinc-approx}
is called the ``DE-Sinc approximation.''
The error analyses
for $i=1$, $2$, $3$, $4$ have been given~\cite{tanaka09:_desinc} as
\begin{align}
 \sup_{t\in (a,\,b)}
\left|
f(t)
-\sum_{k=-n}^n f(\DEt{i}(kh))S(k,h)(\DEtInv{i}(t))
\right|
&\leq C \rme^{-\pi d n/\log(4 d n/\mu)},
\label{leq:DE-Sinc-error}
\intertext{and for $\DEt{3\dagger}(t)$, also given~\cite{tanaka09:_desinc} as}
 \sup_{t\in(0,\,\infty)}
\left|
f(t)
-\sum_{k=-n}^n f(\DEt{3\dagger}(kh))S(k,h)(\DEtInv{3\dagger}(t))
\right|
&\leq C \rme^{-\pi d n/\log(\pi d n/\mu)}.
\label{leq:DE3-Sinc-error}
\end{align}

The main objective of this study is
to refine these error analyses into more useful forms.
It should be emphasized that
the error analyses~\eqref{leq:SE-Sinc-error},
\eqref{leq:DE-Sinc-error},
and~\eqref{leq:DE3-Sinc-error}
are not just asymptotic ($\simeq$),
but are strict inequalities ($\leq$).
Hence, if the constants $C$ are given
in a more explicit, \emph{computable} form,
we can use the term on the right-hand side as a rigorous error bound,
which is quite useful for verified numerical computations.
In fact, the explicit form of $C$ was revealed
in case~4 (the interval is finite)~\cite{okayama09:_error}.
This study reveals the explicit form of $C$
in cases 1--3 (the interval is not finite).

As a second objective, this paper improves
the DE transformation in case~3.
Instead of $\DEt{3}(t)$ or $\DEt{3\dagger}(t)$,
\begin{equation*}
\DEt{3\ddagger}(t)=\log(1 + \rme^{\pi\sinh t})
\end{equation*}
is employed in this paper,
which was originally proposed
for numerical integration~\cite{okayama12:_error}.
The error is also given as
\begin{equation}
 \sup_{t\in(0,\,\infty)}
\left|
f(t)
-\sum_{k=-n}^n f(\DEt{3\ddagger}(kh))S(k,h)(\DEtInv{3\ddagger}(t))
\right|
\leq C \rme^{-\pi d n/\log(2 d n/\mu)},
\label{leq:DE3ddagger-Sinc-error}
\end{equation}
with the explicit constant $C$.
The convergence rate of~\eqref{leq:DE3ddagger-Sinc-error}
is better than that of~\eqref{leq:DE-Sinc-error}
or~\eqref{leq:DE3-Sinc-error}.
Furthermore, in the same manner as
for numerical integration~\cite{okayama12:_error},
it can be shown that
$\DEt{3\ddagger}$ is the best
possible variable transformation in case~3
(although this is not discussed in this paper).

The remainder of this paper is organized as follows.
In Section~\ref{sec:Sinc-approx-estimates},
existing error analyses and
new explicit error bounds (the main result) are stated.
Numerical examples are presented in Section~\ref{sec:numer-exam}.
Proofs of all the theorems
stated in this paper are given in Section~\ref{sec:proofs}.
Section~\ref{sec:conclusion} draws together our conclusions.

\section{Existing error analyses and new error estimates with explicit constants}
\label{sec:Sinc-approx-estimates}

In this section,
after reviewing existing results,
new error bounds for the SE- and DE-Sinc approximations are stated.
Let us first introduce some notation.
Let $\domD_d$ be a strip domain
defined by $\domD_d=\{\zeta\in\mathbb{C}:|\Im \zeta|< d\}$
for $d>0$.
Furthermore, let
$\domD_d^{-}=\{\zeta\in\domD_d :\Re\zeta< 0\}$
and
$\domD_d^{+}=\{\zeta\in\domD_d :\Re\zeta\geq 0\}$.
In this section,
$d$ is assumed to be a positive constant with $d<\pi/2$.
Let $\psi(\domD_d)$ denote the image of $\domD_d$ given by the map $\psi$,
i.e.,
$\psi(\domD_d)=\left\{z=\psi(\zeta): \zeta\in\domD_d \right\}$.
%
Let $I_1=(-\infty,\,\infty)$, $I_2=I_3=(0,\,\infty)$, and
let us define the following three functions:
\begin{align*}
E_1(z;\gamma)&=\frac{1}{(1+z^2)^{\gamma/2}},\\
E_2(z;\alpha,\beta)&=\frac{z^{\alpha}}{(1+z^2)^{(\alpha+\beta)/2}},\\
E_3(z;\alpha,\beta)&=\left(\frac{z}{1+z}\right)^{\alpha}\rme^{-\beta z}.
\end{align*}
We write $E_i(z;\gamma,\gamma)$ as $E_i(z;\gamma)$ for short.

\subsection{Existing error analyses and new error estimates for the SE-Sinc approximation}
\label{subsec:main-SE-Sinc-approx}

Existing error analyses for the Sinc approximation combined
with $\SEt{1}$, $\SEt{2}$, and $\SEt{3}$ are written in the following form
(Theorems~\ref{thm:SE1-Sinc} and~\ref{thm:SE2-Sinc}).

\begin{thm}[Stenger~{\cite[Theorem~4.2.5]{stenger93:_numer}}]
\label{thm:SE1-Sinc}
Assume that $f$ is analytic in $\SEt{1}(\domD_d)$, and
that there exist positive constants $K$, $\alpha$, and $\beta$ such that
\begin{align}
 |f(z)|&\leq K |E_1(z;\alpha)|
\label{leq:Sinc-case1-alpha}
\intertext{for all $z\in\SEt{1}(\domD_d^{-})$, and}
 |f(z)|&\leq K |E_1(z;\beta)|
\label{leq:Sinc-case1-beta}
\end{align}
for all $z\in\SEt{1}(\domD_d^{+})$.
Let $\mu=\min\{\alpha,\,\beta\}$,
let $h$ be defined as
\begin{equation}
h=\sqrt{\frac{\pi d}{\mu n}},
\label{eq:Def-SE-h}
\end{equation}
and let $M$ and $N$ be defined as
\begin{equation}
\begin{cases}
M=n,\quad N=\lceil\alpha n/\beta\rceil
 & \,\,\,(\text{if}\,\,\,\mu = \alpha),\\
N=n,\quad M=\lceil\beta n/\alpha\rceil
 &  \,\,\,(\text{if}\,\,\,\mu = \beta).
\end{cases}
\label{eq:Def-SE-Sinc-MN}
\end{equation}
Then, there exists a constant $C_1$, independent of $n$, such that
\begin{equation}
\sup_{t\in I_1}\left|
f(t)
- \sum_{k=-M}^N f(\SEt{1}(kh))S(k,h)(\SEtInv{1}(t))
\right|
\leq C_1\sqrt{n}\rme^{-\sqrt{\pi d \mu n}}.
\label{leq:SE1-Sinc}
\end{equation}
\end{thm}
\begin{thm}[Stenger~{\cite[Theorem~4.2.5]{stenger93:_numer}}]
\label{thm:SE2-Sinc}
The following is true for $i=2$, $3$.
Assume that $f$ is analytic in $\SEt{i}(\domD_d)$, and
that there exist positive constants $K$, $\alpha$, and $\beta$ such that
\begin{equation}
 |f(z)|\leq K | E_i(z;\alpha,\beta) |
\label{leq:Sinc-case2-alpha-beta}
\end{equation}
for all $z\in\SEt{i}(\domD_d)$.
Let $\mu=\min\{\alpha,\,\beta\}$,
let $h$ be defined as in~\eqref{eq:Def-SE-h},
and let $M$ and $N$ be defined as in~\eqref{eq:Def-SE-Sinc-MN}.
Then, there exists a constant $C_i$, independent of $n$, such that
\begin{equation}
\sup_{t\in I_i}
\left|
f(t)
- \sum_{k=-M}^N f(\SEt{i}(kh))S(k,h)(\SEtInv{i}(t))
\right|
\leq C_i \sqrt{n}\rme^{-\sqrt{\pi d \mu n}}.
\label{leq:SE2-Sinc}
\end{equation}
\end{thm}

This paper explicitly estimates the constant $C_i$
in~\eqref{leq:SE1-Sinc} and~\eqref{leq:SE2-Sinc}
as follows.

\begin{thm}
\label{thm:SE1-Sinc-Explicit}
Let the assumptions in Theorem~\ref{thm:SE1-Sinc} be fulfilled.
Furthermore, let $\nu=\max\{\alpha,\,\beta\}$.
Then, inequality~\eqref{leq:SE1-Sinc} holds with
\[
 C_1=\frac{2^{\nu+1}K}{\sqrt{\pi d \mu}}
\left\{
 \frac{2}{\sqrt{\pi d \mu}(1-\rme^{-2\sqrt{\pi d \mu}})\{\cos d\}^{\nu}} + 1
\right\}.
\]
\end{thm}
\begin{thm}
\label{thm:SE2-Sinc-Explicit}
Let the assumptions in Theorem~\ref{thm:SE2-Sinc} be fulfilled.
Then, inequality~\eqref{leq:SE2-Sinc} holds with
\begin{align*}
 C_2&=\frac{2K}{\sqrt{\pi d \mu}}
\left\{
 \frac{2}{\sqrt{\pi d \mu}(1-\rme^{-2\sqrt{\pi d \mu}})\{\cos d\}^{(\alpha+\beta)/2}} + 1
\right\},\\
C_3&=\frac{2 K}{\sqrt{\pi d \mu}}
\left\{
 \frac{2^{1+(\alpha+\beta)/2}}{\sqrt{\pi d \mu}(1-\rme^{-2\sqrt{\pi d \mu}})\{\cos(d/2)\}^{\alpha+\beta}} + 1
\right\}.
\end{align*}
\end{thm}

\subsection{Existing error analyses and new error estimates for the DE-Sinc approximation}
\label{subsec:main-DE-Sinc}

Existing error analyses for the Sinc approximation
with $\DEt{1}$, $\DEt{2}$, $\DEt{3}$, and $\DEt{3\dagger}$
are written in the following form
(Theorems~\ref{thm:DE1-Sinc} and~\ref{Thm:DEt3-Sinc}).

\begin{thm}[Tanaka et al.~{\cite[Theorems 3.2, 3.3, 3.5]{tanaka09:_desinc}}]
\label{thm:DE1-Sinc}
The following is true for $i=1$, $2$, $3$.
Assume that $f$ is analytic in $\DEt{i}(\domD_d)$, and
that there exist positive constants $K$ and $\mu$
such that
$
|f(z)|\leq K |E_i(z;\mu)|
$
for all $z\in\DEt{i}(\domD_d)$.
Then, there exists a constant $C_i$, independent of $n$, such that
\begin{equation*}
\sup_{t\in I_i}
\left|
f(t)
- \sum_{k=-n}^n f(\DEt{i}(kh))S(k,h)(\DEtInv{i}(t))
\right|
\leq C_i \rme^{-\pi d n/\log(4 d n/\mu)},
\end{equation*}
where
\begin{equation}
h=\frac{\log(4 d n/\mu)}{n}.
\label{eq:Def-DE-h}
\end{equation}
\end{thm}

\begin{rem}
Strictly speaking, in case 3, the range of $d$ is not $0<d<\pi/2$
(assumed at the beginning of this section), but
$0<d<\Im \zeta_0\simeq 1.4045$, where
$\zeta_0=\arcsinh(2\imnum + \frac{2}{\pi}\log(1-\frac{1}{\rme}))$.
This is because the denominator of $E_3(z;\alpha,\beta)$ is zero
at $z = \DEt{3}(\zeta_0)$.
To address this issue,
this paper employs the function $z^{\mu}\rme^{-\mu z}$
in Theorem~\ref{thm:DE3-Sinc-explicit}
instead of $E_3(z;\mu)$.
\end{rem}
\begin{thm}[Tanaka et al.~{\cite[Theorem 3.4]{tanaka09:_desinc}}]
\label{Thm:DEt3-Sinc}
Assume that $f$ is analytic in $\DEt{3\dagger}(\domD_d)$, and
that there exist positive constants $K$ and $\mu$
such that
$|f(z)|\leq K | E_3(z;\mu)|$
for all $z\in\DEt{3\dagger}(\domD_d)$.
Let $h$ be defined as $h=\log(\pi d n/\mu)/n$.
Then, there exists a constant $C$, independent of $n$,
such that~\eqref{leq:DE3-Sinc-error} holds.
\end{thm}


As for case 1 (Theorem~\ref{thm:DE1-Sinc} with $i=1$)
and case 2 (Theorem~\ref{thm:DE1-Sinc} with $i=2$),
this paper not only explicitly estimates the constants $C_i$,
but also generalizes the approximation formula
from $\sum_{k=-n}^n$ to $\sum_{k=-M}^N$ as stated below.

\begin{thm}
\label{thm:DE1-Sinc-explicit}
Assume that $f$ is analytic in $\DEt{1}(\domD_d)$, and
that there exist positive constants $K$, $\alpha$, and $\beta$
such that~\eqref{leq:Sinc-case1-alpha} holds
for all $z\in\DEt{1}(\domD_d^{-})$,
and~\eqref{leq:Sinc-case1-beta} holds
for all $z\in\DEt{1}(\domD_d^{+})$.
Let $\mu=\min\{\alpha,\,\beta\}$,
let $\nu=\max\{\alpha,\,\beta\}$,
let $h$ be defined as in~\eqref{eq:Def-DE-h},
and let $M$ and $N$ be defined as
\begin{equation}
\begin{cases}
M=n,\quad N=n - \lfloor\log(\beta/\alpha)/h\rfloor
 & \,\,\,(\text{if}\,\,\,\mu = \alpha),\\
N=n,\quad M=n-\lfloor\log(\alpha/\beta)/h\rfloor
 &  \,\,\,(\text{if}\,\,\,\mu = \beta).
\end{cases}
\label{eq:Def-DE-Sinc-MN}
\end{equation}
Furthermore,
let $n$ be taken sufficiently large so that
$n\geq (\nu \rme)/(4 d)$ holds.
Then, it holds that
\begin{equation*}
\sup_{t\in I_1}
\left|
f(t)
- \sum_{k=-M}^N f(\DEt{1}(kh))S(k,h)(\DEtInv{1}(t))
\right|
\leq C_1\rme^{-\pi d n/\log(4 d n/\mu)},
\end{equation*}
where $C_1$ is a constant independent of $n$, expressed as
\[
 C_1=\frac{2^{\nu+1}K}{\pi d\mu}
\left\{
 \frac{4}{\pi(1-\rme^{-\pi\mu\rme/2})\{\cos(\frac{\pi}{2}\sin d)\}^{\nu}\cos d}
 + \mu \rme^{\pi\nu/4}
\right\}.
\]
\end{thm}
\begin{thm}
\label{thm:DE2-Sinc-explicit}
Assume that $f$ is analytic in $\DEt{2}(\domD_d)$, and
that there exist positive constants $K$, $\alpha$, and $\beta$
such that~\eqref{leq:Sinc-case2-alpha-beta} holds with $i=2$
for all $z\in\DEt{2}(\domD_d)$.
Let $\mu=\min\{\alpha,\,\beta\}$,
let $\nu=\max\{\alpha,\,\beta\}$,
let $h$ be defined as in~\eqref{eq:Def-DE-h},
and let $M$ and $N$ be defined as in~\eqref{eq:Def-DE-Sinc-MN}.
Furthermore,
let $n$ be taken sufficiently large so that
$n\geq (\nu \rme)/(4 d)$ holds.
Then, it holds that
\begin{equation*}
\sup_{t\in I_2}
\left|
f(t)
- \sum_{k=-M}^N f(\DEt{2}(kh))S(k,h)(\DEtInv{2}(t))
\right|
\leq C_2\rme^{-\pi d n/\log(4 d n/\mu)},
\end{equation*}
where $C_2$ is a constant independent of $n$, expressed as
\[
 C_2=\frac{2K}{\pi d \mu}
\left\{
 \frac{4}{\pi(1-\rme^{-\pi\mu\rme/2})\{\cos(\frac{\pi}{2}\sin d)\}^{(\alpha+\beta)/2}\cos d}
 + \mu\rme^{\pi\nu/4}
\right\}.
\]
\end{thm}

As for case 3 (Theorem~\ref{thm:DE1-Sinc} with $i=3$
and Theorem~\ref{Thm:DEt3-Sinc}),
this paper employs the improved variable transformation
$\DEt{3\ddagger}$ as described in the introduction, and
gives the error estimates
in a similar form to Theorems~\ref{thm:DE1-Sinc-explicit}
and~\ref{thm:DE2-Sinc-explicit}.
\begin{thm}
\label{thm:DE3-Sinc-explicit}
Assume that $f$ is analytic in $\DEt{3\ddagger}(\domD_d)$, and
that there exist positive constants $K$ and $\mu$ ($\mu\leq 1$)
such that $|f(z)|\leq K |z^\mu \rme^{-\mu z}|$ holds
for all $z\in\DEt{3\ddagger}(\domD_d)$.
Let $h$ be defined as
\begin{equation}
h=\frac{\log(2 d n/\mu)}{n}.
\label{eq:Def-DE-h-half}
\end{equation}
Furthermore,
let $n$ be taken sufficiently large so that
$n\geq (\mu \rme)/(2 d)$ holds.
Then, inequality~\eqref{leq:DE3ddagger-Sinc-error} holds with
\[
 C=\frac{K}{\pi^{1-\mu}d \mu}
\left\{
 \frac{4}{\pi(1-\rme^{-\pi\mu\rme})\{\cos(\frac{\pi}{2}\sin d)\}^{2\mu}\{\cos d\}^{\mu+1}}
 + 
\mu 2^{1-\mu}\rme^{\mu(\pi+2)/2}
\right\}.
\]
\end{thm}
\begin{rem}
In Theorem~\ref{thm:DE3-Sinc-explicit},
the sum of the approximation formula is not $\sum_{k=-M}^N$,
but $\sum_{k=-n}^n$.
This is because the inequality bounding the function $f$ is
not $|f(z)|\leq \tilde{K} |z^\alpha \rme^{-\beta z}|$, but
$|f(z)|\leq K |z^\mu \rme^{-\mu z}|$,
which assumes $\alpha=\beta$.
This assumption, however, is made
without loss of generality, because
by setting $z=(\alpha/\beta)w$
and $g(w)=f((\alpha/\beta)w)$ in the former inequality, we obtain
$|g(w)|\leq \tilde{K} (\alpha/\beta)^\alpha |w^\alpha \rme^{-\alpha w}|$.
\end{rem}
\begin{rem}
Theorems~\ref{thm:SE1-Sinc}--\ref{thm:SE2-Sinc-Explicit}
assume that $f$ is analytic in a simply connected domain.
In contrast,
Theorems~\ref{thm:DE1-Sinc}--\ref{thm:DE3-Sinc-explicit}
assume that $f$ is analytic in an infinitely sheeted Riemann surface.
These two assumptions are different;
see Tanaka et al.~\cite[Figures 1--9]{tanaka09:_desinc} for details.
\end{rem}

\section{Numerical Examples} \label{sec:numer-exam}

In this section, some numerical results are presented.
All computation programs were written in C with
double-precision floating-point arithmetic.
The programs and computation results are available at
\texttt{https://github.com/okayamat/sinc-errorbound-infinite}.

First, let us consider the following three examples.

\begin{exam}[Case 1: function decays algebraically as $t\to\pm\infty$]
\label{Exam:case_1}
Consider the function
\[
 f_1(t)=\frac{1}{1+t^2}\sqrt{1+\tanh^2(\arcsinh t)},
\]
which satisfies the assumptions
in Theorem~\ref{thm:SE1-Sinc-Explicit}
with $\alpha=\beta=2$, $d=\pi/4$, and $K=3/2$,
and also
those in Theorem~\ref{thm:DE1-Sinc-explicit}
with $\alpha=\beta=2$, $d=\pi/6$, and $K=3/2$.
\end{exam}
\begin{exam}[Case 2: function decays algebraically as $t\to\infty$]
\label{Exam:case_2}
Consider the function
\[
 f_2(t)=\frac{\sqrt{t}}{1+t^2}\sqrt{1+\tanh^2(\log t)},
\]
which satisfies the assumptions
in Theorem~\ref{thm:SE2-Sinc-Explicit} ($i=2$)
with $\alpha=1/2$, $\beta=3/2$, $d=\pi/4$, and $K=3/2$,
and also
those in Theorem~\ref{thm:DE2-Sinc-explicit}
with $\alpha=1/2$, $\beta=3/2$, $d=\pi/6$, and $K=3/2$.
\end{exam}
\begin{exam}[Case 3: function decays exponentially as $t\to\infty$]
\label{Exam:case_3}
Consider the function
\[
 f_3(t)=t^{\pi/4}\rme^{-t},
\]
which satisfies the assumptions
in Theorem~\ref{thm:SE2-Sinc-Explicit} ($i=3$)
with $\alpha=\pi/4$, $\beta=3/4$, $d=3.14/2$, and $K=(1+(\pi/2)^2)^{\pi/8}$,
and also
those in Theorem~\ref{thm:DE1-Sinc} ($i=3$)
with $\mu=\pi/4$, $d=1.40$, and some constant $K>0$
(note that no explicit error bound is given).
In addition, if we set $g(u)=f_3((\pi/4)u)$,
$g$ satisfies the assumptions in Theorem~\ref{thm:DE3-Sinc-explicit}
with $\mu=\pi/4$, $d=3/2$, and $K=(\pi/4)^{\pi/4}$.
\end{exam}

Numerical results are shown in
Figures~\ref{Fig:example1}--\ref{Fig:example3}.
In Figure~\ref{Fig:example1},
``Maximum Error'' denotes the maximum value
of the absolute error investigated at the following 403 points:
$t=0$, $\pm 2^{-50}$, $\pm 2^{-49.5},\,\ldots$,
$\pm 2^{-0.5}$, $\pm 2^{0}$,
$\pm 2^{0.5},\,\ldots$, $\pm 2^{50}$.
Similarly, Figures~\ref{Fig:example2} and~\ref{Fig:example3}
present the maximum errors at 201 points
(positive points of those stated above).
In each graph,
we can see that the error estimate by the presented theorem (dotted line)
clearly bounds the actual error (solid line).
Furthermore, in Figure~\ref{Fig:example3},
we can confirm that the proposed DE-Sinc approximation
using $\DEt{3\ddagger}$ converges faster than
the old DE-Sinc approximation using $\DEt{3}$ (``old DE'' in the graph).
Another DE transformation $\DEt{3\dagger}$ is not employed, because
$\DEtInv{3\dagger}$ is not easily computed.

Next, let us consider the following example,
which is an unfavorable case for the DE-Sinc approximation.

\begin{exam}[Case 1: function decays algebraically as $t\to\pm\infty$]
\label{Exam:case_1-2}
Consider the function
\[
 f_4(t)=\frac{1}{1+t^2}\sqrt{\cos(3\arcsinh t) + \cosh(\pi)},
\]
which satisfies the assumptions
in Theorem~\ref{thm:SE1-Sinc-Explicit}
with $\alpha=\beta=2$, $d=\pi/3$, and $K=2\cosh(\pi)$,
whereas it does not satisfy the assumptions
in Theorem~\ref{thm:DE1-Sinc-explicit}
as we cannot find any $d>0$ such that
$f_4$ is analytic in $\DEt{1}(\domD_d)$
(we find $\alpha=\beta=2$, though).
\end{exam}

According to Okayama et al.~\cite{okayama13:_de_sinc_same},
even in this case, the DE-Sinc approximation works
almost as well as the SE-Sinc approximation
by choosing $d=\arcsin(d_{\textSE{}}/\pi)$,
where $d_{\textSE{}}$ denotes $d$ in
the SE-Sinc approximation.
Actually,
as shown in Figure~\ref{Fig:example4}, the DE-Sinc approximation
converges at a similar rate to that of the SE-Sinc approximation.
However, note that the error in the DE-Sinc approximation cannot
be estimated by Theorem~\ref{thm:DE1-Sinc-explicit}, whereas
Theorem~\ref{thm:SE1-Sinc-Explicit} works well.

\begin{figure}
\begin{center}
 \begin{minipage}{0.49\linewidth}
  \includegraphics[width=\linewidth]{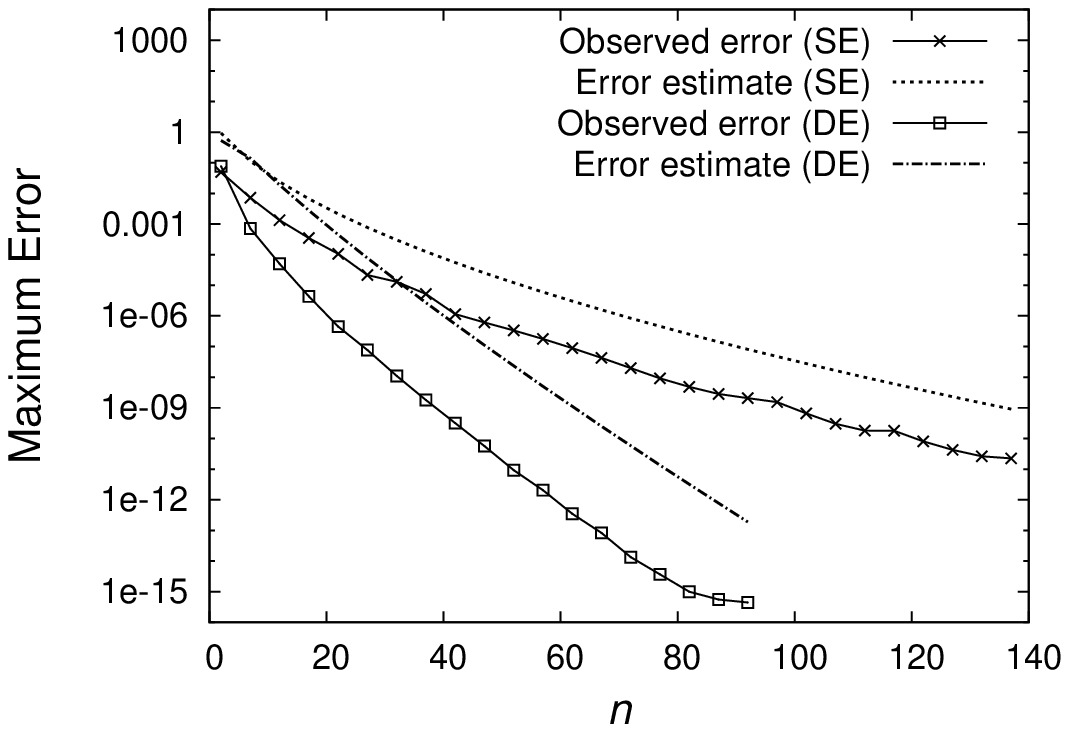}
  \caption{Approximation errors of $f_1$ and their estimates.}
  \label{Fig:example1}
 \end{minipage}
 \begin{minipage}{0.005\linewidth}
 \mbox{ }
 \end{minipage}
 \begin{minipage}{0.49\linewidth}
  \includegraphics[width=\linewidth]{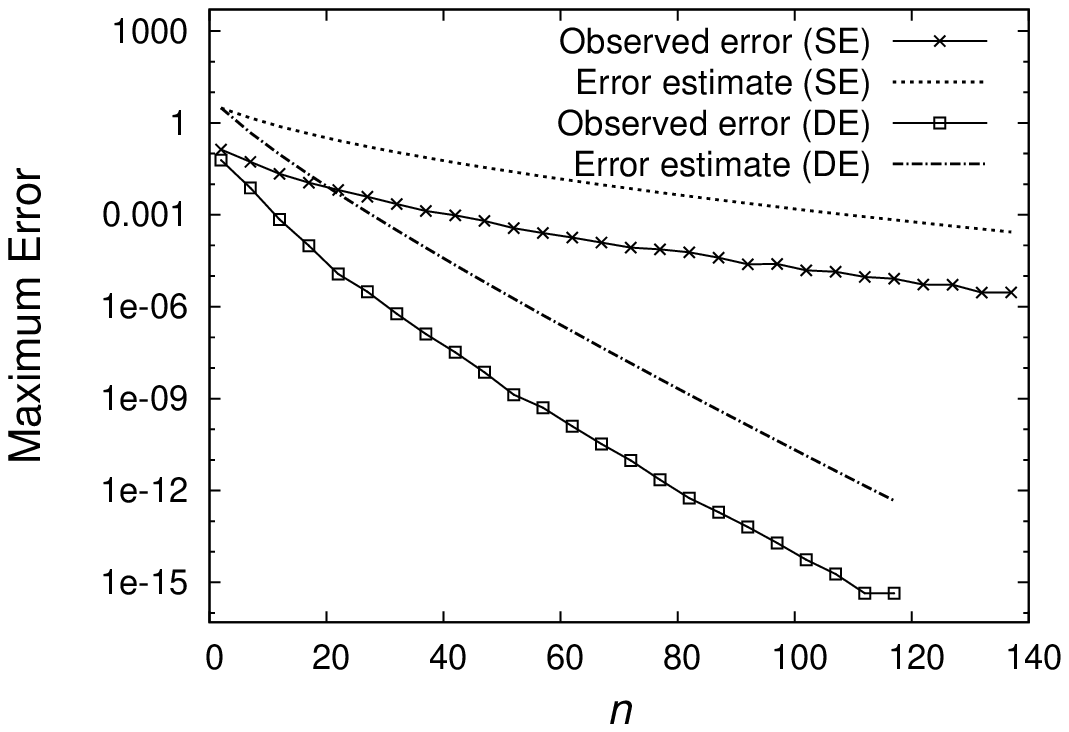}
  \caption{Approximation errors of $f_2$ and their estimates.}
  \label{Fig:example2}
 \end{minipage}

 \begin{minipage}{0.49\linewidth}
  \includegraphics[width=\linewidth]{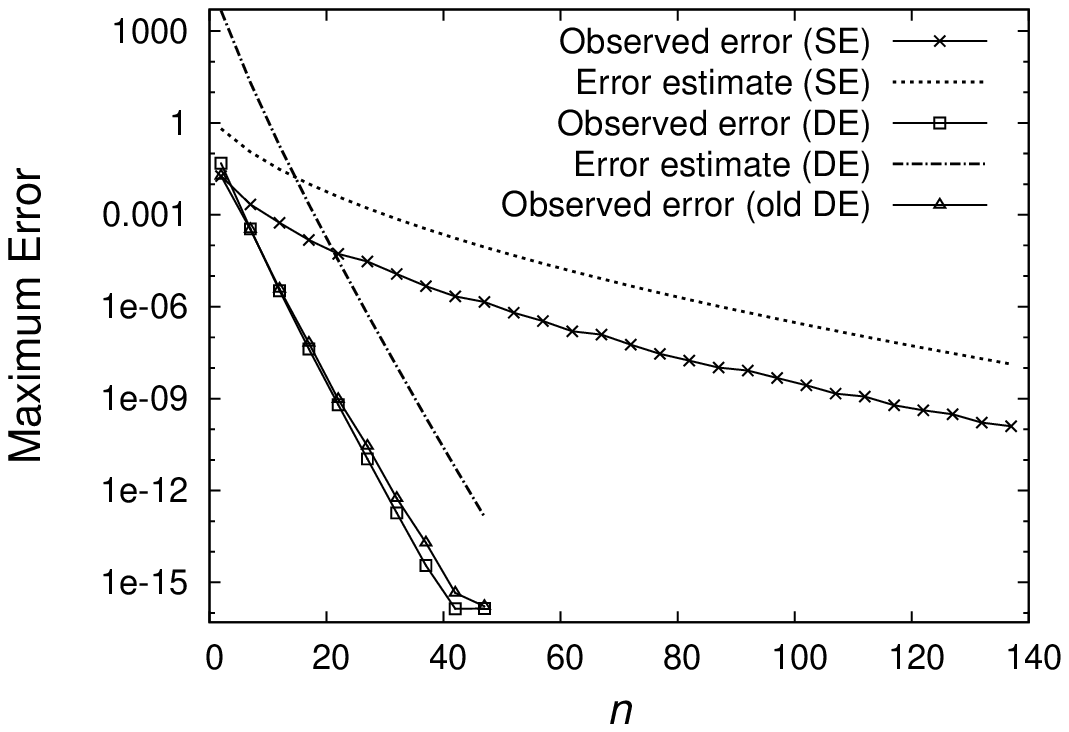}
  \caption{Approximation errors of $f_3$ and their estimates.}
  \label{Fig:example3}
 \end{minipage}
 \begin{minipage}{0.005\linewidth}
 \mbox{ }
 \end{minipage}
 \begin{minipage}{0.49\linewidth}
  \includegraphics[width=\linewidth]{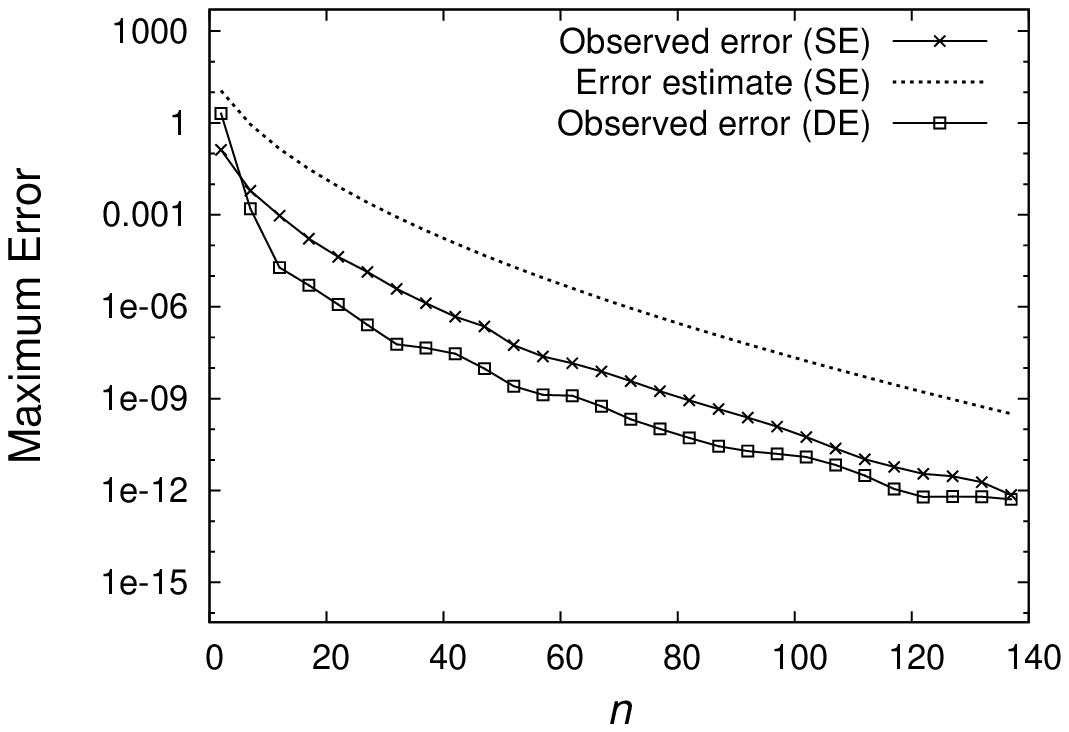}
  \caption{Approximation errors of $f_4$ and their estimates.}
  \label{Fig:example4}
 \end{minipage}
\end{center}
\end{figure}


\section{Proofs}
\label{sec:proofs}

\subsection{In the case of the SE-Sinc approximation}

First, the error of the SE-Sinc approximation~\eqref{eq:SE-Sinc-approx}
is estimated.
Let us look at a sketch of the proof.
Let $F(x)=f(\SEt{i}(x))$
(recall that we employ $t=\SEt{i}(x)$ for the variable transformation),
and evaluate the following term:
\begin{align*}
\left|
f(t)
-\sum_{k=-M}^N f(\SEt{i}(kh))S(k,h)(\SEtInv{i}(t))
\right|
=\left|
F(x)
-\sum_{k=-M}^N F(kh) S(k,h)(x)
\right|,
\end{align*}
which is the error of the Sinc approximation~\eqref{eq:Sinc-approx}.
For the estimation,
the following function space plays an important role.
\begin{defn}
Let $L,\,R,\,\alpha,\,\beta,\,\gamma$ be positive constants,
and let $d$ be a constant with $0<d<\pi/\gamma$.
Then, $\LC_{L,R,\alpha,\beta,\gamma}^{\textSEg}(\domD_d)$
denotes a family of functions $F$
that are analytic in $\domD_d$,
and for all $\zeta\in\domD_d$ and $x\in\mathbb{R}$, satisfy
\begin{align}
 |F(\zeta)|&\leq
 \frac{L}{|1+\rme^{-\gamma\zeta}|^{\alpha/\gamma}|1+\rme^{\gamma\zeta}|^{\beta/\gamma}},
\label{ineq:LC-SE-complex}\\
 |F(x)|&\leq \frac{R}{(1+\rme^{-\gamma x})^{\alpha/\gamma}(1+\rme^{\gamma x})^{\beta/\gamma}}.
\label{ineq:LC-SE-real}
\end{align}
\end{defn}
If $F\in\LC_{L,R,\alpha,\beta,\gamma}^{\textSEg}(\domD_d)$,
the error of the Sinc approximation is estimated as below.
The proof is omitted here because it is quite similar to
that of the existing theorem for case 4~\cite[Theorem~2.4]{okayama09:_error}.
\begin{thm}
\label{thm:SE-Sinc-overall}
Let $F\in\LC_{L,R,\alpha,\beta,\gamma}^{\textSEg}(\domD_d)$,
let $\mu=\min\{\alpha,\,\beta\}$,
let $h$ be defined as in~\eqref{eq:Def-SE-h},
and let $M$ and $N$ be defined as in~\eqref{eq:Def-SE-Sinc-MN}.
Then, setting $\epsilon_n^{\textSEg}=\sqrt{n}\rme^{-\sqrt{\pi d \mu n}}$,
we have
\begin{equation*}
\sup_{x\in\mathbb{R}}\left|
F(x)
-\sum_{k=-M}^{N} F(kh) S(k,h)(x)
\right|
\leq \frac{2}{\sqrt{\pi d \mu}}
\left[
\frac{2L}{\sqrt{\pi d \mu}(1-\rme^{-2\sqrt{\pi d \mu}})\{\cos(\gamma d/2)\}^{(\alpha+\beta)/\gamma}}
+R
\right]\epsilon_n^{\textSEg}.
\end{equation*}
\end{thm}

In view of Theorem~\ref{thm:SE-Sinc-overall},
the proof is completed by checking
$F\in\LC_{L,R,\alpha,\beta,\gamma}^{\textSEg}(\domD_d)$
for cases 1, 2, and 3.
Let us examine each case individually.

\subsubsection{Proof in case 1 (Theorem~\ref{thm:SE1-Sinc-Explicit})}

The claim of
Theorem~\ref{thm:SE1-Sinc-Explicit} follows
from the next lemma.

\begin{lem}
Let the assumptions in Theorem~\ref{thm:SE1-Sinc-Explicit}
be fulfilled.
Then, the function $F(\zeta)=f(\SEt{1}(\zeta))$
belongs to $\LC_{L,R,\alpha,\beta,\gamma}^{\textSEg}(\domD_d)$
with $L=2^{\nu}K/\{\cos d\}^{(\nu-\mu)/2}$, $R=2^{\nu}K$, and $\gamma=2$.
\end{lem}

The proof is straightforward
by the next result, because
$|F(\zeta)|\leq K|E_1(\SEt{1}(\zeta);\alpha)|$
and
$|F(\zeta)|\leq K|E_2(\SEt{1}(\zeta);\beta)|$ coincide with
the inequalities in the following lemma.

\begin{lem}[Okayama~{\cite[in the proof of Lemma~5.1]{okayama12:_error}}]
Assume that $F$ is analytic in $\domD_d$ with $0<d<\pi/2$,
and that there exist positive constants $K$, $\alpha$, $\beta$
such that
\begin{align*}
|F(\zeta)|&\leq \frac{K}{|1+\rme^{-2\zeta}|^{\alpha/2} |1+\rme^{2\zeta}|^{\alpha/2}}
\intertext{holds for all $\zeta\in\domD_d^{-}$, and}
|F(\zeta)|&\leq \frac{K}{|1+\rme^{-2\zeta}|^{\beta/2} |1+\rme^{2\zeta}|^{\beta/2}}
\end{align*}
holds for all $\zeta\in\domD_d^{+}$.
Then, $F\in\LC_{L,R,\alpha,\beta,\gamma}^{\textSEg}(\domD_d)$
with $L=2^{\nu}K/\{\cos d\}^{(\nu-\mu)/2}$, $R=2^{\nu}K$, and $\gamma=2$,
where $\mu=\min\{\alpha,\,\beta\}$ and $\nu=\max\{\alpha,\,\beta\}$.
\end{lem}

\subsubsection{Proof in case 2 (Theorem~\ref{thm:SE2-Sinc-Explicit} with $i=2$)}

The claim of
Theorem~\ref{thm:SE2-Sinc-Explicit} with $i=2$ follows
from the next lemma.

\begin{lem}
Let the assumptions in Theorem~\ref{thm:SE2-Sinc-Explicit}
be fulfilled with $i=2$.
Then, the function $F(\zeta)=f(\SEt{2}(\zeta))$
belongs to $\LC_{L,R,\alpha,\beta,\gamma}^{\textSEg}(\domD_d)$
with $L=K$, $R=K$, and $\gamma=2$.
\end{lem}
\begin{proof}
From inequality~\eqref{leq:Sinc-case2-alpha-beta} with $i=2$,
\eqref{ineq:LC-SE-complex}
and~\eqref{ineq:LC-SE-real} immediately hold with $L=R=K$ and $\gamma=2$.
\end{proof}

\subsubsection{Proof in case 3
(Theorem~\ref{thm:SE2-Sinc-Explicit} with $i=3$)}

The claim of
Theorem~\ref{thm:SE2-Sinc-Explicit} follows
from the next lemma.

\begin{lem}
\label{lem:SE3-Sinc-Quad-check}
Let the assumptions in Theorem~\ref{thm:SE2-Sinc-Explicit}
be fulfilled with $i=3$.
Then, the function $F(\zeta)=f(\SEt{3}(\zeta))$
belongs to $\LC_{L,R,\alpha,\beta,\gamma}^{\textSEg}(\domD_d)$
with $L=2^{(\alpha+\beta)/2} K$,
$R=K$, and $\gamma=1$.
\end{lem}

For the proof,
let us prepare some useful inequalities
(Lemmas~\ref{lem:asinh-SE} and~\ref{lem:exp-asinh-SE}).

\begin{lem}[Okayama~{\cite[Lemma~5.4]{okayama12:_error}}]
\label{lem:asinh-SE}
For all $\zeta\in\overline{\domD_{\pi/2}}$ and $x\in\mathbb{R}$,
we have
\begin{align*}
\left|
\frac{\arcsinh(\rme^{\zeta})}{1+\arcsinh(\rme^{\zeta})}
\right|
&\leq\sqrt{2}
\left|\frac{\rme^{\zeta}}{1+\rme^{\zeta}}\right|,\\
\frac{\arcsinh(\rme^{x})}{1+\arcsinh(\rme^{x})}
&\leq\frac{\rme^{x}}{1+\rme^{x}}.
\end{align*}
\end{lem}

\begin{lem}[Okayama~{\cite[Lemma~5.5]{okayama12:_error}}]
\label{lem:exp-asinh-SE}
For all $\zeta\in\overline{\domD_{\pi/2}}$
and $x\in\mathbb{R}$, we have
\begin{align*}
\frac{1}{|\rme^{\zeta}+\sqrt{1+\rme^{2\zeta}}|}
&\leq\frac{\sqrt{2}}{|1+\rme^{\zeta}|},\\
\frac{1}{\rme^{x}+\sqrt{1+\rme^{2x}}}
&\leq\frac{1}{1+\rme^{x}}.
\end{align*}
\end{lem}

By using the lemmas above,
Lemma~\ref{lem:SE3-Sinc-Quad-check}
can be proved as follows.

\begin{proof}
From inequality~\eqref{leq:Sinc-case2-alpha-beta} with $i=3$,
it follows that
\begin{align*}
|F(\zeta)|\leq
 K
\left|\frac{\arcsinh(\rme^{\zeta})}{1+\arcsinh(\rme^{\zeta})}\right|^{\alpha}
\left|\frac{1}{\rme^{\zeta}+\sqrt{1+\rme^{2\zeta}}}\right|^{\beta}.
\end{align*}
From Lemmas~\ref{lem:asinh-SE} and~\ref{lem:exp-asinh-SE},
it holds that
\begin{align*}
|F(\zeta)|
&\leq K\left|\frac{\sqrt{2}}{1+\rme^{-\zeta}}\right|^{\alpha}
\left|\frac{\sqrt{2}}{1+\rme^{\zeta}}\right|^{\beta}
= 2^{(\alpha+\beta)/2}K
\left|\frac{1}{1+\rme^{-\zeta}}\right|^{\alpha}
\left|\frac{1}{1+\rme^{\zeta}}\right|^{\beta}
\end{align*}
for all $\zeta\in\domD_d$. For $x\in\mathbb{R}$, it holds that
\begin{align*}
|F(x)|
&\leq K\left(\frac{1}{1+\rme^{-x}}\right)^{\alpha}
\left(\frac{1}{1+\rme^{x}}\right)^{\beta}.
\end{align*}
This completes the proof.
\end{proof}

\subsection{In the case of the DE-Sinc approximation}

Next, the error of the DE-Sinc approximation is estimated.
In this case,
we again estimate the error of the Sinc approximation~\eqref{eq:Sinc-approx}
in view of the following term:
\begin{align*}
\left|
f(t)
-\sum_{k=-M}^N f(\DEt{i}(kh))S(k,h)(\DEtInv{i}(t))
\right|
=\left|
F(x)
-\sum_{k=-M}^N F(kh) S(k,h)(x)
\right|,
\end{align*}
where we set $F(x)=f(\DEt{i}(x))$.
Because of differences in the variable transformation,
instead of $\LC_{L,R,\alpha,\beta,\gamma}^{\textSEg}(\domD_d)$,
we require the following function space.
\begin{defn}
Let $L,\,R,\,\alpha,\,\beta$ be positive constants,
and let $d$ be a constant with $0<d<\pi/2$.
Then, $\LC_{L,R,\alpha,\beta}^{\textDEg}(\domD_d)$
denotes a family of functions $F$
that are analytic in $\domD_d$,
and for all $\zeta\in\domD_d$ and $x\in\mathbb{R}$, satisfy
\begin{align}
 |F(\zeta)|&\leq \frac{L}{|1+\rme^{-\pi\sinh\zeta}|^{\alpha/2}|1+\rme^{\pi\sinh\zeta}|^{\beta/2}},
\label{ineq:LC-DE-complex}\\
 |F(x)|&\leq \frac{R}{(1+\rme^{-\pi\sinh x})^{\alpha/2}(1+\rme^{\pi\sinh x})^{\beta/2}}.
\label{ineq:LC-DE-real}
\end{align}
\end{defn}
If $F\in\LC_{L,R,\alpha,\beta}^{\textDEg}(\domD_d)$,
the errors of the Sinc approximation are estimated as below.
%
%
The proofs are omitted because they are quite similar to
that of the existing theorem for
case~4~\cite[Theorem~2.11]{okayama09:_error}.
\begin{thm}
\label{thm:DE-Sinc-overall}
Let $F\in\LC_{L,R,\alpha,\beta}^{\textDEg}(\domD_d)$,
let $\mu=\min\{\alpha,\,\beta\}$,
let $\nu=\max\{\alpha,\,\beta\}$,
let $h$ be defined as in~\eqref{eq:Def-DE-h},
and let $M$ and $N$ be defined as in~\eqref{eq:Def-DE-Sinc-MN}.
Furthermore,
let $n$ be taken sufficiently large so that
$n\geq (\nu \rme)/(4 d)$ holds.
Then,
setting $\epsilon_n^{\textDEg}=\rme^{-\pi d n/\log(4 d n/\mu)}$,
we have
\begin{align*}
&\sup_{x\in\mathbb{R}}\left|
F(x)
-\sum_{k=-M}^{N} F(kh)S(k,h)(x)
\right|
\leq \frac{2}{\pi d \mu}
\left[
\frac{4L}{\pi(1-\rme^{-\pi\mu\rme/2})\{\cos(\frac{\pi}{2}\sin d)\}^{(\alpha+\beta)/2}\cos d}
+\mu R\rme^{\pi\nu/4}
\right]\epsilon_n^{\textDEg}.
\end{align*}
\end{thm}
In view of Theorem~\ref{thm:DE-Sinc-overall},
our main task here is to check the assumption that
$F\in\LC_{L,R,\alpha,\beta}^{\textDEg}(\domD_d)$.
The next lemma is useful for the proofs.
\begin{lem}[Okayama et al.~{\cite[Lemma~4.22]{okayama09:_error}}]
\label{lem:DEfunc-estim}
Let $x,\,y\in\mathbb{R}$ with $|y|<\pi/2$.
Then, it holds that
\begin{align*}
\left|\frac{1}{1+\rme^{\pi\sinh(x+\imnum y)}}\right|
&\leq\frac{1}{(1+\rme^{\pi\sinh(x)\cos y})\cos(\frac{\pi}{2}\sin y)},\\
\left|\frac{1}{1+\rme^{-\pi\sinh(x+\imnum y)}}\right|
&\leq\frac{1}{(1+\rme^{-\pi\sinh(x)\cos y})\cos(\frac{\pi}{2}\sin y)}.
\end{align*}
\end{lem}

\subsubsection{Proof in case 1 (Theorem~\ref{thm:DE1-Sinc-explicit})}

The claim of
Theorem~\ref{thm:DE1-Sinc-explicit} follows
from the next lemma.

\begin{lem}
Let the assumptions in Theorem~\ref{thm:DE1-Sinc-explicit}
be fulfilled.
Then, the function $F(\zeta)=f(\DEt{1}(\zeta))$
belongs to $\LC_{L,R,\alpha,\beta}^{\textDEg}(\domD_d)$
with $L=2^{\nu}K/\{\cos(\frac{\pi}{2}\sin d)\}^{(\nu-\mu)/2}$
and $R=2^{\nu}K$.
\end{lem}

The proof is straightforward
by the next result, because
$|F(\zeta)|\leq K|E_1(\DEt{1}(\zeta);\alpha)|$
and
$|F(\zeta)|\leq K|E_2(\DEt{1}(\zeta);\beta)|$ coincide with
the inequalities in the following lemma.

\begin{lem}[Okayama~{\cite[in the proof of Lemma~5.9]{okayama12:_error}}]
Assume that $F$ is analytic in $\domD_d$ with $0<d<\pi/2$,
and that there exist positive constants $K$, $\alpha$, $\beta$
such that
\begin{align*}
|F(\zeta)|&\leq \frac{K}{|1+\rme^{-\pi\sinh\zeta}|^{\alpha/2} |1+\rme^{\pi\sinh\zeta}|^{\alpha/2}}
\intertext{holds for all $\zeta\in\domD_d^{-}$, and}
|F(\zeta)|&\leq \frac{K}{|1+\rme^{-\pi\sinh\zeta}|^{\beta/2} |1+\rme^{\pi\sinh\zeta}|^{\beta/2}}
\end{align*}
holds for all $\zeta\in\domD_d^{+}$.
Then, $F\in\LC_{L,R,\alpha,\beta}^{\textDEg}(\domD_d)$
with $L=2^{\nu}K/\{\cos(\frac{\pi}{2}\sin d)\}^{(\nu-\mu)/2}$
and $R=2^{\nu}K$,
where $\mu=\min\{\alpha,\,\beta\}$ and $\nu=\max\{\alpha,\,\beta\}$.
\end{lem}

\subsubsection{Proof in case 2 (Theorem~\ref{thm:DE2-Sinc-explicit})}

The claim of
Theorem~\ref{thm:DE2-Sinc-explicit} follows
from the next lemma.

\begin{lem}
Let the assumptions in Theorem~\ref{thm:DE2-Sinc-explicit}
be fulfilled.
Then, the function $F(\zeta)=f(\DEt{2}(\zeta))$
belongs to $\LC_{L,R,\alpha,\beta}^{\textDEg}(\domD_d)$
with $L=K$ and $R=K$.
\end{lem}
\begin{proof}
From inequality~\eqref{leq:Sinc-case2-alpha-beta} with $i=2$,
\eqref{ineq:LC-DE-complex}
and~\eqref{ineq:LC-DE-real} immediately hold with $L=R=K$.
\end{proof}

\subsubsection{Proof in case 3 (Theorem~\ref{thm:DE3-Sinc-explicit})}

In contrast to cases 1 and 2,
the function $F(\zeta)=f(\DEt{3\ddagger}(\zeta))$
does not belong to $\LC_{L,R,\alpha,\beta}^{\textDEg}(\domD_d)$
under the assumptions of
Theorem~\ref{thm:DE3-Sinc-explicit}.
In this case, we split the error into two terms as
\begin{align*}
&\left|
F(x) - \sum_{k=-n}^n F(kh) S(k,h)(x)
\right|\\
&\leq
\left|
F(x) - \sum_{k=-\infty}^{\infty} F(kh) S(k,h)(x)
\right|
+
\left|
\sum_{k=-\infty}^{-n-1} F(kh) S(k,h)(x)
+
\sum_{k=n+1}^{\infty} F(kh) S(k,h)(x)
\right|,
\end{align*}
which are called
the ``discretization error''
and the ``truncation error,'' respectively.
Let us estimate these two terms separately.
For the discretization error,
we require the following function space.
\begin{defn}
Let $\domD_d(\epsilon)$ be a rectangular domain defined
for $0<\epsilon<1$ by
\[
\domD_d(\epsilon)
= \{\zeta\in\mathbb{C}:|\Re\zeta|<1/\epsilon,\, |\Im\zeta|<d(1-\epsilon)\}.
\]
Then, $\Hone(\domD_d)$ denotes the family of all functions $F$
that are analytic in $\domD_d$ such that the norm $\mathcal{N}_1(F,d)$ is finite, where
\[
\mathcal{N}_1(F,d)
=\lim_{\epsilon\to 0}\oint_{\partial \domD_d(\epsilon)} |F(\zeta)||\divv \zeta|.
\]
\end{defn}

The discretization error for
a function $F$ belonging to $\Hone(\domD_d)$
has been estimated as follows.

\begin{thm}[Stenger~{\cite[Theorem~3.1.3]{stenger93:_numer}}]
\label{Thm:Sinc-Infinite-Sum-Approx}
Let $F\in\Hone(\domD_d)$. Then,
\[
    \sup_{x\in\mathbb{R}}
    \left|F(x)-\sum_{k=-\infty}^{\infty}F(kh)S(k,h)(x)\right|
\leq\frac{\mathcal{N}_1(F,d)}{\pi d(1-\rme^{-2\pi d/h})}\rme^{-\pi d/h}.
\]
\end{thm}

This paper shows that $F\in\Hone(\domD_d)$
under the assumptions of Theorem~\ref{thm:DE3-Sinc-explicit}.

\begin{lem}
\label{lem:DE3-Sinc-check}
Let the assumptions in Theorem~\ref{thm:DE3-Sinc-explicit}
be fulfilled.
Then, the function $F(\zeta)=f(\DEt{3\ddagger}(\zeta))$
belongs to $\Hone(\domD_d)$, and $\mathcal{N}_1(F,d)$ is estimated as
\[
 \mathcal{N}_1(F,d)
\leq
\frac{4 K}{\pi^{1-\mu}\mu \cos^{2\mu}(\frac{\pi}{2}\sin d)\cos^{1+\mu} d}.
\]
\end{lem}

For the proof, the following inequality is useful.

\begin{lem}
\label{lem:log-DE}
For all real numbers $x$ and $y$ with $|y|<\pi/2$,
we have
\[
|\log(1+\rme^{\pi \sinh(x+\imnum y)})|
\leq \frac{1}{\cos(\frac{\pi}{2}\sin y)\cos y}
\cdot\frac{\pi\cosh x}{1 + \rme^{-\pi\sinh(x)\cos y}}.
\]
\end{lem}

As the proof is relatively long, it is given at the end of this section.
If we accept this lemma,
Lemma~\ref{lem:DE3-Sinc-check} can be proved as follows.
\begin{proof}
By assumption, clearly $F$ is analytic in $\domD_d$.
In the following, we estimate $\mathcal{N}_1(F,d)$.
Using Lemmas~\ref{lem:DEfunc-estim} and~\ref{lem:log-DE},
we have
\begin{align*}
|F(x+\imnum y)|&=|f(\DEt{3\ddagger}(x+\imnum y))|\\
&\leq K\left|\log(1+\rme^{\pi\sinh(x+\imnum y)})\right|^{\mu}
\left|\frac{1}{1+\rme^{\pi\sinh(x+\imnum y)}}\right|^{\mu}\\
&\leq K \left\{
\frac{1}{\cos(\frac{\pi}{2}\sin y)\cos y}\cdot
\frac{\pi\cosh x}{1+\rme^{-\pi\sinh(x)\cos y}}
\right\}^{\mu}
\left\{
\frac{1}{\cos(\frac{\pi}{2}\sin y)}\cdot
\frac{1}{1+\rme^{\pi\sinh(x)\cos y}}
\right\}^{\mu}\\
&=\frac{K \pi^{\mu}}{\cos^{2\mu}(\frac{\pi}{2}\sin y)\cos^{\mu} y}
\cdot\frac{\cosh^{\mu} x}
{(1+\rme^{-\pi\sinh(x)\cos y})^{\mu}(1+\rme^{\pi\sinh(x)\cos y})^{\mu}}\\
&\leq
\frac{K \pi^{\mu}}{\cos^{2\mu}(\frac{\pi}{2}\sin y)\cos^{\mu} y}
\cdot\frac{\cosh x}{\rme^{\pi\mu\sinh(|x|)\cos y}}.
\end{align*}
From this, for any $\epsilon$ with $0<\epsilon<1$, it holds that
\[
  \lim_{x\to\pm\infty}
\int_{-d(1-\epsilon)}^{d(1-\epsilon)}
|F(x+\imnum y)| d y
\leq  \lim_{x\to\pm\infty}
\frac{K \pi^{\mu}\cosh x}
{\rme^{\pi\mu\sinh(|x|)\cos d(1-\epsilon)}}
\int_{-d(1-\epsilon)}^{d(1-\epsilon)}
\frac{\diff y}{\cos^{2\mu}(\frac{\pi}{2}\sin y)\cos^{\mu} y}
=0.
\]
Therefore, we have
\begin{align*}
 \mathcal{N}_1(F,d)
&=\lim_{y\to d}\int_{-\infty}^{\infty}
|F(x+\imnum y)|\diff x
+\lim_{y\to -d}\int_{-\infty}^{\infty}
|F(x+\imnum y)|\diff x\\
&\leq
\frac{2 K \pi^{\mu}}{\cos^{2\mu}(\frac{\pi}{2}\sin d)\cos^{\mu} d}
\int_{-\infty}^{\infty}
\frac{\cosh x}{\rme^{\pi\mu\sinh(|x|)\cos d}}\diff x
=\frac{4 K \pi^{\mu}}{\cos^{2\mu}(\frac{\pi}{2}\sin d)\cos^{\mu} d}
\cdot
\frac{1}{\pi \mu \cos d},
\end{align*}
which is finite if $0<d<\pi/2$. This completes the proof.
\end{proof}

Combining the above results,
we have the following estimate for the discretization error.

\begin{lem}\label{lem:DE3-Sinc-discretize}
Let the assumptions in Theorem~\ref{thm:DE3-Sinc-explicit}
be fulfilled. Then, setting $F(\zeta)=f(\DEt{3\ddagger}(\zeta))$,
we have
\[
    \sup_{x\in\mathbb{R}}
    \left|F(x)-\sum_{k=-\infty}^{\infty}F(kh)S(k,h)(x)\right|
\leq\frac{4K}{\pi^{2-\mu}d \mu(1-\rme^{-2\pi d/h})\cos^{2\mu}(\frac{\pi}{2}\sin d)\cos^{1+\mu} d}\rme^{-\pi d/h}.
\]
\end{lem}

The truncation error is estimated as follows.

\begin{lem}\label{lem:DE3-Sinc-truncate}
Let the assumptions in Theorem~\ref{thm:DE3-Sinc-explicit}
be fulfilled. Then, setting $F(\zeta)=f(\DEt{3\ddagger}(\zeta))$,
we have
\begin{align*}
\left|
\sum_{k=-\infty}^{-n-1} F(kh) S(k,h)(x)
+
\sum_{k=n+1}^{\infty} F(kh) S(k,h)(x)
\right|
\leq \frac{2^{2-\mu}K\rme^{\frac{\pi}{2}\mu}}{\mu \pi^{1-\mu}h\rme^{nh(1-\mu)}}
\rme^{-\frac{\pi}{2}\mu\exp(nh)}.
\end{align*}
\end{lem}
\begin{proof}
Using $|S(k,h)(x)|\leq 1$, we have
\[
 \left|
\sum_{k=-\infty}^{-n-1} F(kh) S(k,h)(x)
+
\sum_{k=n+1}^{\infty} F(kh) S(k,h)(x)
\right|
\leq
\sum_{k=-\infty}^{-n-1} |F(kh)|
+
\sum_{k=n+1}^{\infty} |F(kh)|.
\]
Furthermore, from Lemma~\ref{lem:log-DE}, it holds that
\begin{align*}
 |F(x)|&\leq K|\log(1+\rme^{\pi\sinh x})|^{\mu}
\left|\frac{1}{1+\rme^{\pi\sinh x}}\right|^{\mu}
\leq K \frac{\pi^{\mu}\cosh^{\mu} x}{(1+\rme^{-\pi\sinh x})^{\mu}}
\cdot
\frac{1}{(1+\rme^{\pi\sinh x})^{\mu}}.
\end{align*}
Therefore, the desired estimate is obtained as
\begin{align*}
\sum_{k=-\infty}^{-n-1} |F(kh)|
+
\sum_{k=n+1}^{\infty} |F(kh)|
&\leq 2\sum_{k=n+1}^{\infty}
\frac{K \pi^{\mu}\cosh^{\mu} (kh)}{(1+\rme^{-\pi\sinh (kh)})^{\mu}(1+\rme^{\pi\sinh (kh)})^{\mu}}\\
&\leq 2 K \pi^{\mu}
\sum_{k=n+1}^{\infty}
\cosh^{\mu}(kh)\rme^{-\pi\mu\sinh(kh)}\\
&\leq \frac{2 K \pi^{\mu}}{h}\int_{nh}^{\infty}\cosh^{\mu}(x)\rme^{-\pi\mu\sinh x}
\diff x\\
&\leq \frac{2 K \pi^{\mu}}{h\cosh^{1-\mu}(nh)}\int_{nh}^{\infty}\cosh(x)\rme^{-\pi\mu\sinh x}
\diff x\\
&=\frac{2 K \pi^{\mu}}{h\cosh^{1-\mu}(nh)}\cdot\frac{\rme^{-\pi \mu \sinh(nh)}}{\pi \mu}
\leq
\frac{2 K \pi^{\mu}}{h (\rme^{nh}/2)^{1-\mu}}\cdot\frac{\rme^{\frac{\pi}{2}\mu}\rme^{-\frac{\pi}{2} \mu \exp(nh)}}{\pi \mu}.
\end{align*}
This completes the proof.
\end{proof}

Using Lemmas~\ref{lem:DE3-Sinc-discretize}
and~\ref{lem:DE3-Sinc-truncate},
we obtain the desired estimate (Theorem~\ref{thm:DE3-Sinc-explicit})
as follows.
\begin{proof}
Let $F(\zeta)=f(\DEt{3\ddagger}(\zeta))$.
Lemmas~\ref{lem:DE3-Sinc-discretize}
and~\ref{lem:DE3-Sinc-truncate} give the following inequality:
\begin{align*}
& \left|
F(x) - \sum_{k=-n}^n F(kh) S(k,h)(x)
\right|\\
&\leq
\frac{4K}{\pi^{2-\mu}d \mu \cos^{2\mu}(\frac{\pi}{2}\sin d)\cos^{1+\mu} d}
\cdot\frac{\rme^{-\pi d/h}}{1-\rme^{-2\pi d/h}}
+\frac{2^{2-\mu}K\rme^{\frac{\pi}{2}\mu}}{\mu \pi^{1-\mu}}
\cdot\frac{\rme^{-\frac{\pi}{2}\mu\exp(nh)}}{h\rme^{nh(1-\mu)}}.
\end{align*}
Furthermore, by substituting the formula in~\eqref{eq:Def-DE-h-half},
the first term can be estimated as
\[
\frac{\rme^{-\pi d/h}}{1-\rme^{-2\pi d/h}}
=\frac{\rme^{-\pi d n/\log(2 d n/\mu)}}{1-\rme^{-\pi \mu (2 d n/\mu)/\log(2 d n/\mu)}}
\leq
\frac{\rme^{-\pi d n/\log(2 d n/\mu)}}{1-\rme^{-\pi \mu \rme/\log(\rme)}},
\]
and the second term can be estimated as
\begin{align*}
\frac{\rme^{-\frac{\pi}{2}\mu\exp(nh)}}{h\rme^{nh(1-\mu)}}
=\frac{\rme^{-\pi d n}}{(\log(2 d n/\mu)/n) (2 d n/\mu)^{1-\mu}}
&=\frac{(2 d n/\mu)^{\mu}\rme^{-\frac{\pi}{2}\mu (2dn/\mu)(1-1/\log(2 d n/\mu))}}{\log(2 d n/\mu)}
\frac{\mu}{2d}\rme^{-\pi d n/\log(2 d n/\mu)}\\
&\leq
\frac{(\rme)^{\mu}\rme^{-\frac{\pi}{2}\mu (\rme)(1-1/\log(\rme))}}{\log(\rme)}
\frac{\mu}{2d}\rme^{-\pi d n/\log(2 d n/\mu)},
\end{align*}
which completes the proof.
\end{proof}

The remaining task is to prove Lemma~\ref{lem:log-DE}.
For the purpose, the following lemma is needed.
\begin{lem}
\label{lem:last-inequality}
For all $x\in\mathbb{R}$, it holds that
\[
\sqrt{x^2+\pi^2}\geq (1+\rme^{-x})\log(1+\rme^x).
\]
\end{lem}
\begin{proof}
Set $f(x)=\rme^x\sqrt{x^2+\pi^2}-(1+\rme^{x})\log(1+\rme^x)$.
The desired inequality is obtained by showing that $f(x)\geq 0$
for $x\in\mathbb{R}$.
We show this separately in two cases: (i) $x\leq 4\pi/3$ and (ii) $x>4\pi/3$.

\noindent
(i) $x\leq 4\pi /3$.
Set $p(x)=\sqrt{x^2+\pi^2}-\log(1+\rme^x)$.
If we can show that $p(x)>1$, we obtain the desired inequality,
as $f(x)=\rme^x p(x)-\log(1+\rme^x)>\rme^x -\log(1+\rme^x)\geq 0$.
Therefore, we show $p(x)>1$ below. We have
\begin{align*}
 p'(x)&= \frac{x}{\sqrt{x^2+\pi^2}} - \frac{\rme^x}{1+\rme^x},\\
p''(x)&= \frac{\pi^2}{(x^2+\pi^2)^{3/2}} - \frac{1}{\{2\cosh(x/2)\}^2}.
\end{align*}
Furthermore, it holds for all $x\in\mathbb{R}$ that
\[
 2\cosh(x/2)\geq 2 + \frac{x^2}{4}> \frac{1}{\pi}(x^2+\pi^2)^{3/4},
\]
from which $p''(x)\geq 0$ holds for all $x\in\mathbb{R}$.
Therefore, $p'(x)$ is a monotonically increasing function.
Noting that $p'(x)\to 0$ as $x\to\infty$, we see that $p'(x)\leq 0$.
Therefore, $p(x)$ is a monotonically decreasing function.
Thus, for $x\leq 4\pi/3$, we have $p(x)\geq p(4\pi/3)>1$.
This completes the proof for $x\leq 4\pi/3$.

\noindent
(ii) $x> 4\pi /3$. Set $\tilde{p}(x)=\sqrt{x^2+3^2}-\log(1+\rme^x)$.
As it holds for all $x\in\mathbb{R}$ that
\[
 2\cosh(x/2)\geq 2 + \frac{x^2}{4}> \frac{1}{3}(x^2+3^2)^{3/4},
\]
we see that $\tilde{p}(x)$ is a monotonically decreasing function
in the same way as above.
Furthermore, because $\tilde{p}(x)\to 0$ as $x\to\infty$,
$\tilde{p}(x)\geq 0$ holds for all $x\in\mathbb{R}$.
Therefore, we have
\begin{align*}
p(x)&=\sqrt{x^2+\pi^2} - \sqrt{x^2+3^2} + \tilde{p}(x)\\
&\geq \sqrt{x^2+\pi^2} - \sqrt{x^2+3^2} + 0
=\frac{\pi^2-3^2}{\sqrt{x^2+\pi^2}+\sqrt{x^2+3^2}}\\
&\geq
\frac{\pi^2-3^2}{\sqrt{x^2+\pi^2}+\sqrt{x^2+\pi^2}}
\geq
\frac{\pi^2-3^2}{\sqrt{x^2+x^2}+\sqrt{x^2+x^2}}
=\frac{\pi^2-3^2}{2\sqrt{2}x}.
\end{align*}
The last inequality uses $x\geq \pi$.
From this inequality,
setting $q(x)=\rme^{x}(\pi^2-3^2)/(2\sqrt{2}x)$, we have
\[
f(x)=\rme^x p(x) - \log(1+\rme^x)
\geq q(x) - \log(1+\rme^x).
\]
Therefore, the desired inequality is proved if we can show
that $q(x)> \log(1+\rme^x)$ for $x> 4\pi /3$.
Here, we see that $q(x)$ is a convex function for $x>0$, because
\begin{align*}
 q'(x)&=\frac{\rme^x(\pi^2-3^2)(x-1)}{2\sqrt{2}x^2},\\
q''(x)&=\frac{\rme^x(\pi^2-3^2)\{(x-1)^2+1\}}{2\sqrt{2}x^3} > 0.
\end{align*}
Therefore, considering
a tangent line at $x=4\pi/3$, we have
\[
 q(x)
\geq q'(4\pi/3)(x-(4\pi/3)) + q(4\pi/3)
\]
for $x>0$. Furthermore, it holds for $x\geq 4\pi/3$ that
\[
 \log(1+\rme^x) = x + \log(1+\rme^{-x})
\leq  x + \log(1+\rme^{-4\pi/3})
= \left(x-(4\pi/3)\right)+(4\pi/3)
+ \log(1+\rme^{-4\pi/3}).
\]
Thus, from $q'(4\pi/3)> 1$ and $q(4\pi/3)>(4\pi/3) + \log(1+\rme^{-4\pi/3})$,
it holds for $x> 4\pi/3$ that
\[
 q'(4\pi/3)(x-(4\pi/3)) + q(4\pi/3)
>
1\cdot(x-(4\pi/3))+\left\{(4\pi/3)+\log(1+\rme^{4\pi/3})\right\},
\]
which shows $q(x)> \log(1+\rme^x)$.
This completes the proof for $x>4\pi/3$.
\end{proof}

Based on this lemma, we can prove Lemma~\ref{lem:log-DE} as follows.

\begin{proof}
Using $|\cosh(x+\imnum y)|\leq \cosh x$,
Lemma~\ref{lem:DEfunc-estim}, and
\[
 \frac{\divv}{\divv z}\log(1+\rme^{\pi\sinh z})
=\frac{\pi\cosh z}{1+\rme^{-\pi\sinh z}},
\]
we have
\begin{align*}
|\log(1+\rme^{\pi\sinh(x+\imnum y)})|
=\left|\int_{-\infty}^x
 \frac{\pi\cosh(t+\imnum y)}{1+\rme^{-\pi\sinh(t+\imnum y)}}\diff t
\right|
&\leq \int_{-\infty}^x
  \frac{\pi|\cosh(t+\imnum y)|}{|1+\rme^{-\pi\sinh(t+\imnum y)}|}\diff t\\
&\leq \int_{-\infty}^x
  \frac{\pi\cosh t}{(1+\rme^{-\pi\sinh(t)\cos y})\cos(\frac{\pi}{2}\sin y)}
\diff t\\
&=\frac{1}{\cos(\frac{\pi}{2}\sin y)\cos y}\log(1+\rme^{\pi\sinh(x)\cos y}).
\end{align*}
The desired inequality is obtained by showing that
\[
\log(1+\rme^{\pi\sinh(x)\cos y})
\leq \frac{\pi\cosh x}{1+\rme^{-\pi\sinh(x)\cos y}},
\]
which is equal to
\[
 \frac{(1+\rme^{-t})\log(1+\rme^t)}{\sqrt{\pi^2+(t/\cos y)^2}}\leq 1,
\]
where $t=\pi\sinh(x)\cos y$. This inequality can be shown as
\[
 \frac{(1+\rme^{-t})\log(1+\rme^t)}{\sqrt{\pi^2+(t/\cos y)^2}}
\leq
\frac{(1+\rme^{-t})\log(1+\rme^t)}{\sqrt{\pi^2+t^2}}
\leq 1,
\]
where Lemma~\ref{lem:last-inequality} has been used in the last inequality.
This completes the proof.
\end{proof}


\section{Concluding remarks}
\label{sec:conclusion}

The errors of the SE-Sinc approximation
over the infinite and semi-infinite intervals
have been analyzed in Theorems~\ref{thm:SE1-Sinc} and~\ref{thm:SE2-Sinc},
where the constant $C_i$ is not given in an explicit form.
This paper revealed the explicit form of each $C_i$,
enabling us to obtain a rigorous, \emph{mathematically correct}
error bound by computing the right-hand side
of~\eqref{leq:SE1-Sinc} or~\eqref{leq:SE2-Sinc}.
These results are useful for computations with guaranteed accuracy.
Similarly, for the DE-Sinc approximation, this paper presented
Theorems~\ref{thm:DE1-Sinc-explicit}--\ref{thm:DE3-Sinc-explicit},
where the error bounds are given in a computable form.
In case 3, in particular,
although existing studies employed
$\DEt{3}(t)$ or $\DEt{3\dagger}(t)$ as a variable transformation,
we used $\DEt{3\ddagger}(t)$, which improves
the convergence rate of the DE-Sinc approximation.
Note that the error bounds are valid if the assumptions are satisfied,
otherwise we cannot use the error bounds
(see Figure~\ref{Fig:example4}, for example, where the assumptions
of Theorem~\ref{thm:DE1-Sinc-explicit} are not satisfied).

Future work includes the following.
First, in addition to the four cases listed in Section~\ref{sec:introduction},
other cases and corresponding conformal maps
have been considered by Stenger~\cite[Section~1.5.3]{stenger11:_handb}.
Computable error bounds for these cases should be given.
Furthermore, the DE-Sinc approximation for these cases should also be
considered. Second, the SE-Sinc and DE-Sinc approximations
described in this paper assume that $f(t)\to 0$ as $t\to a$
and $t\to b$.
However, there are more general approximation formulas 
that can handle the case where the boundary values of $f$ are
non-zero~\cite{stenger11:_handb,okayama13:_sinc}.
Computable error bounds for these cases are desirable.

\bibliography{Sinc-error-estim-infinite}

\end{document}